\documentclass[reqno,a4paper,oneside]{amsart}
\usepackage[T1]{fontenc}
\usepackage[utf8]{inputenc}
\usepackage[english]{babel}
\usepackage{microtype}
\usepackage[paperwidth=7in, paperheight=10in, margin=.875in]{geometry}
\usepackage{emptypage}
\usepackage{bm}
\usepackage{enumerate}
\usepackage{amsfonts}
\usepackage{amsmath}
\usepackage{amsthm}
\usepackage{amssymb}
\usepackage{mathtools}
\usepackage{mathrsfs}
\usepackage{dsfont}
\usepackage{graphicx}
\usepackage{subfig}
\usepackage{caption}
\usepackage{cases}
\usepackage{lipsum}
\usepackage{afterpage}
\usepackage[dvipsnames]{xcolor}
\usepackage{listings}
\usepackage{hyperref}
\DeclarePairedDelimiter{\abs}{\lvert}{\rvert}
\DeclarePairedDelimiter{\norm}{\lVert}{\rVert}

\newcommand{\real}{\mathbb{R}}

\newcommand{\nat}{\mathbb{N}}

\newcommand{\prob}{\mathbb{P}}
\newcommand{\E}{\mathbb{E}}

\newcommand{\barr}[2][3]{{}\mkern#1mu\overline{\mkern-#1mu#2}}

\newcommand{\de}{\mathrm{d}}
\newcommand{\ev}{\mathrm{ev}}

\newcommand{\barX}{\barr{X}}

\newcommand{\barB}{\barr{B}}

\newcommand{\barR}{\barr{R}}
\newcommand{\barp}{\barr{p}}

\newcommand{\tv}{\tilde{v}}
\newcommand{\tsigma}{\tilde{\sigma}}
\newcommand{\tL}{\tilde{L}}

\newcommand{\gm}{\mathfrak{m}}

\newcommand{\cC}{\mathcal{C}}
\newcommand{\cM}{\mathcal{M}}
\newcommand{\cX}{\mathcal{X}}
\newcommand{\cY}{\mathcal{Y}}
\newcommand{\cP}{\mathcal{P}}
\newcommand{\cT}{\mathcal{T}}

\newcommand{\cK}{\mathcal{K}}

\newcommand{\cS}{\mathcal{S}}
\newcommand{\cA}{\mathcal{A}}

\newcommand{\cO}{\mathcal{O}}

\newcommand{\sX}{\mathscr{X}}
\newcommand{\sY}{\mathscr{Y}}
\newcommand{\sF}{\mathscr{F}}

\newcommand{\Law}{\operatorname{Law}}

\newcommand{\diag}{\operatorname{diag}}
\newcommand{\trace}{\operatorname{tr}}
\newcommand{\argmin}{\operatorname{argmin}}
\newcommand{\supp}{\operatorname{supp}}

\numberwithin{equation}{section}

\theoremstyle{plain}
\newtheorem{thm}{Theorem}[section]
\newtheorem{cor}[thm]{Corollary}
\newtheorem{lemma}[thm]{Lemma}
\newtheorem{prop}[thm]{Proposition}
\theoremstyle{definition}
\newtheorem{defn}[thm]{Definition}
\theoremstyle{remark}
\newtheorem{rem}[thm]{Remark}

    \title[An alternative approach to well-posedness of mean-field CBO]{An alternative approach to well-posedness of {M}c{K}ean-{V}lasov equations arising in \\ consensus-based optimization}

    \author[A. Baldi]{Alessandro Baldi}
    \address[Alessandro Baldi]{Dipartimento di Scienze Matematiche ``G. L. Lagrange'', Politecnico di Torino, Corso Duca degli Abruzzi, 24, 10129 Torino, Italy. ORCID: 0009-0006-4778-8431.}
    \email{alessandro.baldi@polito.it}
    
    \keywords{Consensus-Based Optimization, McKean--Vlasov equations, multi-agent systems, interacting particle systems, mean-field limit.}

    \subjclass[2020]{60K35, 65C35, 93A16, 70F45, 35Q93.}

    \date{\today}
    
\begin{document}

    \begin{abstract} In this work we study the mean-field description of Consensus-Based Optimization (CBO), a derivative-free particle optimization method. Such a description is provided by a non-local SDE of McKean--Vlasov type, whose fields lack of global Lipschitz continuity. We propose a novel approach to prove the well-posedness of the mean-field CBO equation based on a truncation argument. The latter is performed through the introduction of a cut-off function, defined on the space of probability measures, acting on the fields. This procedure allows us to study the well-posedness problem in the classical framework of Sznitman. Through this argument, we recover the established result on the existence of strong solutions, and we extend the class of solutions for which pathwise uniqueness holds. 
    \end{abstract}

    \maketitle

    \allowdisplaybreaks
          
\section{Introduction}\label{intro}
Consensus-Based Optimization (CBO) is a particle optimization method, first introduced in \cite{CBO-Originale}, designed to address global optimization problems of the form
\begin{equation}\label{eq:optimization}
    \min_{x \in \real^d} f(x),
\end{equation}
where $f \colon \real^d \to \real$ is a, possibly non-convex and non-differentiable, continuous \emph{objective function}. The CBO method consists of a population of $N \in \nat_+$ agents, whose microscopic state at time $t \in [0,T]$ is determined by their positions $X^i_t \in \real^d$, $i = 1,\dots,N$, on the domain of the objective function $f$. The initial positions $X^i_0$, $i=1,\dots,N$, of the particles are assumed to be independent and identically distributed with law $\rho_0 \in \cP(\real^d)$. The time evolution of the system is described by the following It\={o} stochastic differential equations
\begin{equation}\label{eq:N_particle_CBO}
    \de X^i_t = -\lambda\bigl(X^i_t - \cM_\beta(\rho^N_t)\bigr)\,\de t + S\bigl(X^i_t - \cM_\beta(\rho^N_t)\bigr)\,\de B^i_t,\qquad i = 1,\dots,N,
\end{equation}
for $t \in [0,T]$, where $T > 0$ is a fixed time horizon. The processes $(B^i_t)_{t \in [0,T]}$, $i = 1,\dots,N$, are $d$-dimensional independent standard Brownian motions, and $\lambda > 0$ is a fixed drift parameter. The point $\cM_\beta(\rho^N_t)$ describes the \emph{instantaneous consensus point} of the population, and it is defined as the convex combination of the positions of the agents
\begin{equation}\label{eq:consensus_point_particle}
    \cM_\beta(\rho^N_t) \coloneqq \frac{\int_{\real^d} x e^{-\beta f(x)}\,\de\rho^N_t(x)}{\int_{\real^d} e^{-\beta f(x)}\,\de\rho^N_t(x)} =
    \sum_{i=1}^N\Biggl( \frac{e^{-\beta f(X^i_t)} }{\sum_{j=1}^N e^{-\beta f(X^j_t)}}\Biggr) X^i_t,
\end{equation}
where $\beta > 0$ is a fixed parameter, and the (random) measure $\rho^N_t = \frac{1}{N} \sum_{j=1}^N \delta_{X^j_t}$ is the \emph{empirical measure} of the system. The instantaneous consensus point encodes the interaction among particles, and it is defined in such a way that the agents located at positions where 
$f$ attains lower values contribute more to the weighted average in \eqref{eq:consensus_point_particle}. 
Furthermore, it acts as an instantaneous guess for the global minimizer of $f$; its form is justified by the fact that, when $\beta$ is large, \eqref{eq:consensus_point_particle} provides a good approximation of $\argmin_{i=1,\dots,N}f(X^i_t)$.
The map $S\colon\real^d\to\real^{d\times d}$ modulates the diffusive effects, and we assume that it is globally Lipschitz continuous and satisfies the condition $S(0) = 0_{d \times d}$, where $0_{d \times d}$ is the zero matrix. This choice includes, as special cases, both the original \emph{isotropic} CBO \cite{CBO-Originale}, where $S(u) = \sqrt{2\theta}\norm{u} I_d$ (with $\theta > 0$), and the \emph{anisotropic} CBO \cite{CBOCarrilloAnisotropic}, where $S(u) = \sqrt{2\theta}\diag(u)$. 

The stochastic evolution \eqref{eq:N_particle_CBO} features a drift term driving particles towards the consensus point, and a diffusive term allowing them to explore the state space and, at the same time, to escape local minima. The aim of the method is to induce concentration (consensus) of the agents around a global minimizer of $f$ after a sufficiently long time. In the computational practice, an Euler--Maruyama time discretization of \eqref{eq:N_particle_CBO} is typically implemented. See the review \cite{CBOTotzeckReview} for more details about the ideas and motivations underlying the CBO method, as well as variants and scenarios of application. We also mention that, since its introduction, the CBO paradigm has given rise to a variety of methods aimed at tackling diverse classes of optimization problems, including, but not limited to, constrained optimization \cite{CBOConstrained, CBOSphere, CBOSphereAnisotropic}, optimal control \cite{CBOHertyModPredContr}, and multi-objective optimization \cite{CBOHertyMultiObj, CBOTotzeckMultiObj}; see also \cite{fornasier2026CBOevolutionstrategies, fornasier2025regularitypositivityCBO} for comprehensive reviews of applications.

The CBO method exhibits two remarkable features: it achieves very good performance in computational optimization \cite{CBOCarrilloAnisotropic, CBOTotzeckReview}, even when the dimension $d$ of the state space is large, and it is amenable to a rigorous theoretical study by means of techniques pertaining to the mathematical analysis of interacting multi-agent systems. A powerful tool in the study of particles interacting through their empirical measure is \emph{mean-field analysis}. One allows the number $N$ of agents to go to infinity, and seeks to approximate the coupled particle dynamics \eqref{eq:N_particle_CBO} with the statistical behaviour of a representative particle. The \emph{mean-field equation} for the CBO system, which describes the time evolution of the representative particle $\barX_t$, reads
\begin{equation}\label{eq:mean_field_CBO}
\begin{dcases}
    \de \barX_t = -\lambda\bigl(\barX_t - \cM_\beta(\rho_t)\bigr)\,\de t + S\bigl(\barX_t - \cM_\beta(\rho_t)\bigr)\,\de \barB_t, \\
    \rho_t = \Law(\barX_t),
\end{dcases}
\end{equation}
for $t \in [0,T]$, subject to a $\rho_0$-distributed initial state $\barX_0$. Here, the consensus point is 
\begin{equation}\label{eq:consensus_point}
    \cM_\beta(\rho_t) = \frac{\int_{\real^d} x e^{-\beta f(x)}\,\de\rho_t(x)}{\int_{\real^d} e^{-\beta f(x)}\,\de\rho_t(x)}.
\end{equation}
The mean-field equation \eqref{eq:mean_field_CBO} is obtained by formally replacing the empirical measure $\rho^N_t$ in \eqref{eq:N_particle_CBO} with the measure $\rho_t = \Law(\barX_t)$, thus uncoupling the interactions in the particle system. A rigorous justification of this formal procedure amounts to proving a \emph{propagation of chaos} result. This was first achieved through a compactness argument in \cite{CBOHuangFirstPropChaos}; stronger quantitative results have been obtained in the recent works by means of a synchronous coupling approach, which allows one to obtain uniform-in-time estimates; see \cite{gerber2026uniformintimepropagationchaosCBO, MeanFieldHoff, CBOKaliseJump} and references therein.

In the mean-field model, the fact that \eqref{eq:consensus_point} provides a reasonable approximation for the global minimizer of $f$ is heuristically corroborated by the well-known \emph{Laplace principle} \cite{DemboZeitouni,MillerAsint,CBO-Originale}, according to which, if $\rho \in \cP(\real^d)$ is absolutely continuous,
\begin{equation*}
    \lim_{\beta\to+\infty}\biggl(-\frac{1}{\beta}\log\biggl(\int_{\real^d}e^{-\beta f(x)}\,\de \rho(x)\biggr)\biggr) = \inf_{x \in \supp(\rho)} f(x),
\end{equation*}
suggesting that, when $f$ admits a unique global minimizer $x_*$, the weighted measure $e^{-\beta f(x)} \rho/(\int_{\real^d} e^{-\beta f(x)} \de\rho(x))$ is a proxy for $\delta_{x_*}$ for large values of $\beta$, so that $\cM_\beta(\rho) \approx x_*$.

The mean-field SDE \eqref{eq:mean_field_CBO} is of \emph{McKean--Vlasov type}, since its fields depend on the law of the solution itself. The study of \eqref{eq:mean_field_CBO} (and of its associated Fokker--Planck equation) allows to deduce the long-time behaviour of the CBO method at an aggregate level: the works \cite{CBO-Carrillo, CBOCarrilloAnisotropic, CBOFornasierConvergenceAnisotropic, CBOFornasierConvergence} establish conditions under which the mean-field law $\rho_t$ concentrates around the global minimizers of the objective function $f$, thus providing theoretical guarantees for the efficacy of the CBO methodology.

The present work is devoted to the presentation of a novel technique to prove the well-posedness of the mean-field CBO equation \eqref{eq:mean_field_CBO}. This result was originally achieved in the seminal work \cite{CBO-Carrillo} through an application of the Leray--Schauder fixed-point theorem in the space of curves $C([0,T],\real^d)$. Here, we propose an alternative strategy based on a \emph{cut-off} function defined on the space of probability measures, by means of which we suitably truncate the fields driving equation \eqref{eq:mean_field_CBO}. The truncated fields turn out to be globally Lipschitz continuous, so that the existence of strong solutions to the associated McKean--Vlasov SDE can be obtained within the classical framework of Sznitman \cite{Sznitman}. Finally, the truncation is relaxed until a strong solution to the original mean-field CBO equation \eqref{eq:mean_field_CBO} is recovered.

Moreover, the present argument allows us to improve the pathwise uniqueness results available in the literature. Indeed, we prove that pathwise uniqueness holds in the class of strong solutions for which the map $t \mapsto \cM_\beta(\rho_t)$ is bounded over $[0,T]$, thus relaxing the continuity requirement on the aforementioned map (see, e.g., \cite[Theorem 2.3]{MeanFieldHoff} for comparison). We now present the statement of the main result.

\begin{thm}\label{thm:MainResult} 
Let $f \in \cO(s,\ell)$, with $s,\ell\ge0$, and $p \ge 2 \vee p_\cM(s,\ell)$ (see Definition \ref{def:ClassO_sl} and \eqref{eq:p_M} below). Given a $d$-dimensional standard Brownian motion $\barB$ defined on a filtered probability space $(\Omega, \sF, (\sF_t)_{t\in[0,T]},\prob)$, and an $\sF_0$-measurable random variable $\barX_0\colon\Omega\to\real^d$ belonging to $L^p(\Omega,\sF, \prob)$, there exists a strong solution $\barX \colon \Omega \to C\bigl([0,T],\real^d\bigr)$ to \eqref{eq:mean_field_CBO} with initial state $\barX_0$. The solution $\barX$ is pathwise unique in the class of strong solutions to \eqref{eq:mean_field_CBO} such that the map $t \mapsto \cM_\beta(\rho_t)$ is bounded over $[0,T]$.
\end{thm}

We remark that, while the core content of Theorem~\ref{thm:MainResult} is well-known in the literature, the main goal of this work is to illustrate an alternative technical possibility in the study of the well-posedness of mean-field equations arising in CBO-type models.

\subsection{Plan of the paper.} This paper is structured as follows. In Section~\ref{sec:WellPosednessArguments}, we recall Sznitman's well-posedness argument \cite{Sznitman} for McKean-Vlasov-type SDEs with globally Lipschitz fields, we review the established approach for the CBO equation \eqref{eq:mean_field_CBO}, and we outline our alternative approach. In Section~\ref{sec:StructuralAssumptions}, we present our structural assumptions on the objective function $f$, together with their consequences on the structure of the fields driving \eqref{eq:mean_field_CBO}; here, we recall the relevant stability and sublinearity estimates available in the literature. The argument for the existence of solutions via truncation is presented in Section~\ref{sec:Existence}, whereas Section~\ref{sec:Uniqueness} is devoted to pathwise uniqueness. We conclude the paper with Appendix~\ref{app:SzinitmanProof}, where, for the reader's convenience, we present in detail the version of Sznitman's argument employed in this work.

\subsection{Notation and preliminaries.}\label{subsec:Notation} 
Given a measure space $(\cX, \sX, \mu)$, a measurable space $(\cY, \sY)$, and a measurable function $g \colon \cX \to \cY$, we denote by $g_\sharp \mu$ the \emph{push-forward of $\mu$ through $g$}, defined as $g_\sharp \mu(B) \coloneqq \mu(g^{-1}(B))$, for all $B \in \sY$. By construction, $g_\sharp \mu$ is a measure on $(\cY, \sY)$ having the same mass of $\mu$. When $g \colon \Omega \to \cY$ is a random variable defined on a probability space $(\Omega, \sF, \prob)$, we denote its law $g_\sharp \prob$ by $\Law(g)$. Given a metric space $(\cX, d_\cX)$, we indicate with $\cP(\cX)$ the convex set of Borel probability measures on $\cX$. Moreover, for all $p \ge 1$, we denote with $\cP_p(\cX)$ the collection of all probability measures with finite moment of order $p$, namely
\begin{equation*}
    \cP_p(\cX) \coloneqq \biggl\{ \mu \in \cP(\cX)\;:\, \int_{\cX} d_{\cX}(x,x_0)^p \, \de\mu(x) < +\infty \text{ for some } x_0 \in \cX \biggr\}.
\end{equation*}
For all fixed $p \ge 1$, the set $\cP_p(\cX)$ can be endowed with a metric structure by means of the \emph{Wasserstein distance of order} $p$, defined as
\begin{equation*}
    W_p(\mu_1, \mu_2) \coloneqq \biggl(\inf_{\pi \in \Gamma(\mu_1, \mu_2)} \int_{\cX \times \cX} d_\cX(x_1, x_2)^p\,\de\pi(x_1, x_2) \biggr)^{1/p}
\end{equation*}
for all $\mu_1,\mu_2 \in \cP_p(\cX)$. The set $\Gamma(\mu_1, \mu_2)$ is the collection of \emph{transport plans} (or \emph{couplings}) between the measures $\mu_1$ and $\mu_2$, which is defined as
\begin{equation*}
    \Gamma(\mu_1, \mu_2) \coloneqq \Bigl\{\pi \in \cP(\cX^2) \,:\, \pi(A\times \cX) = \mu_1(A),\; \pi(\cX \times B) = \mu_2(B),\, \text{ for all Borel } A,B\Bigr\}.
\end{equation*}
If the metric space $(\cX, d_\cX)$ is separable and complete, so is $(\cP_p(\cX),W_p)$. Throughout this work, we shall primarily consider the cases where $\cX$ is either the Euclidean space $\real^d$ or the space of continuous functions $C([0,T], \real^d)$ (equipped with the uniform norm), both of which are separable and complete. For more details on the Wasserstein spaces, we refer the reader to the texts \cite{Optimal-Transport-Ambrosio, AGS, Optimal-Transport-Santambrogio, Optimal-Transport-Villani}.

For the sake of conciseness, by denoting with $\norm{x}$ the Euclidean norm of the vector $x\in \real^d$, we adopt the notation
\begin{equation}\label{eq:momentum_def}
    \gm_p(\mu) \coloneqq \biggl(\int_{\real^d} \norm{x}^p\,\de\mu(x)\biggr)^{1/p}, \qquad \text{for all }\mu \in \cP_p(\real^d),
\end{equation}
and, for all fixed $R>0$ and $p \ge 1$, we define the class of measures 
\begin{equation}\label{eq:P_pR}
    \cP_{p,R}(\real^d) \coloneqq \Bigl\{\mu \in \cP(\real^d)\,\, \colon\, \gm_p(\mu) \le R\Bigr\}.
\end{equation}

On the vector space of real square matrices $\real^{d\times d}$, we shall consider the standard Frobenius norm $\norm{\cdot}_F$. If $H \in \real^{d\times d}$ is symmetric and positive semidefinite, we denote by $\sqrt{H}$ its \emph{matrix square root}, namely, the unique symmetric and positive semidefinite matrix $Q \in \real^{d\times d}$ satisfying $Q^\top Q = H$ (for further details, see, e.g., \cite[Chapter 6]{BellmanMatrix}).

We shall indicate by $\ev_t$ the \emph{evaluation map} at time $t$, namely, the map $C([0,T],\real^d) \to \real^d$ such that $\ev_t(\gamma) \coloneqq \gamma(t)$.
The symbols $\nabla \psi$, $H_\psi$, and $\Delta \psi$ will respectively denote the gradient, Hessian matrix, and Laplace operator of a function $\psi\colon\real^d \to \real$.

\section{Strategies for the proof of well-posedness of McKean--Vlasov SDEs}\label{sec:WellPosednessArguments}
Here, we recall the established techniques for the proof of existence and pathwise uniqueness of strong solutions to McKean--Vlasov SDEs with globally Lipschitz fields, as well as the special techniques developed for the CBO mean-field equation \eqref{eq:mean_field_CBO}, where global Lipschitz continuity of the fields is no longer ensured. We conclude the section by presenting an outline of the alternative approach devised in this work, together with the main result.

Throughout this section, we shall  consider a McKean--Vlasov SDE of the form
\begin{equation}\label{eq:McKean-Vlasov}
\begin{dcases}
\de \barX_t = v_{\rho_t}(\barX_t)\,\de t + \sigma_{\rho_t}(\barX_t) \,\de \barB_t\,, \\
\rho_t = \Law(\barX_t),
\end{dcases} 
\qquad\text{ for } t \in [0,T],
\end{equation}
where $\barB$ is a $d$-dimensional standard Brownian motion defined on a filtered probability space $(\Omega, \sF, (\sF_t)_{t\in[0,T]},\prob)$. The dynamics is driven by the fields $v \colon \real^d \times \cP(\real^d) \to \real^d$ and $\sigma \colon \real^d \times \cP(\real^d) \to \real^{d\times d}$, and the initial state of the system is assigned by an $\sF_0$-measurable random variable $\barX_0\colon\Omega\to\real^d$. 

\subsection{The classical proof of Sznitman in the case of globally Lipschitz fields.}\label{subsec:ClassicalSznitman} The study of well-posedness of McKean--Vlasov equations dates back to the work of McKean himself \cite{McKean1967}. Two decades later, in his lecture notes \cite{Sznitman}, Sznitman introduced a technique based on a fixed-point argument on the space of probability measures over the path space. This approach relies on the global Lipschitz  continuity of the fields $v$ and $\sigma$ with respect to both space and measure arguments. More precisely, Sznitman's argument requires the existence of $\barp \ge 1$, and of positive constants $L_v, L_\sigma$, such that, for all $x_1,x_2 \in \real^d$, $\mu_1,\mu_2\in\cP_{\barp}(\real^d)$, 
\begin{subequations}\label{eq:GlobalLip}
\begin{align}
    \norm{v_{\mu_1}(x_1) - v_{\mu_2}(x_2)} &\le L_v\bigl(\norm{x_1 - x_2} + W_{\barp}(\mu_1, \mu_2) \bigr), \label{eq:GlobalLipDrift} \\
    \norm{\sigma_{\mu_1}(x_1) - \sigma_{\mu_2}(x_2)}_F &\le L_\sigma\bigl(\norm{x_1 - x_2} + W_{\barp}(\mu_1, \mu_2) \bigr). \label{eq:GlobalLipDiffCBO}
\end{align}
\end{subequations}
Under the additional assumption that the initial state $\barX_0$ for problem \eqref{eq:McKean-Vlasov} belongs to $L^p(\Omega,\sF,\prob)$, with $p \ge 2\vee\barp$, one can consider a fixed measure $\Psi \in \cP_p(C([0,T],\real^d))$ and formulate the auxiliary stochastic differential equation (with $\barX_0$ as initial state)
\begin{equation}\label{eq:AuxProblem}
\de X^\Psi_t = v_{\Psi_t}(X^\Psi_t)\,\de t + \sigma_{\Psi_t}(X^\Psi_t) \,\de \barB_t\,, \qquad\text{ for } t \in [0,T],   
\end{equation}
where $\Psi_t \coloneqq (\ev_t)_\sharp \Psi$, for all $t \in [0,T]$; by recalling the definition of \emph{push-forward} given at the beginning of Section~\ref{subsec:Notation}, since $\Psi$ is probability measure in $\cP_p(C([0,T],\real^d))$, then one has that $\Psi_t$ is a probability measure belonging to $\cP_p(\real^d)$ for all $t \in [0,T]$.
Then, as a consequence of \eqref{eq:GlobalLip}, one can show that the map $\cS \colon \cP_p(C([0,T],\real^d)) \to \cP_p(C([0,T],\real^d))$ 
\begin{equation}\label{eq:DefS}
\Psi \mapsto \cS(\Psi) \coloneqq \Law(X^\Psi)
\end{equation}
(where $X^\Psi$ is the pathwise unique strong solution to \eqref{eq:AuxProblem}) is well-defined and admits a unique fixed point $\rho \in \cP_p(C([0,T],\real^d))$ through the Banach--Caccioppoli fixed-point theorem (see, \emph{e.g.},\cite{BanachHistory2024}). Since by construction $\cS(\rho) = \Law(X^\rho) = \rho$, taking the push-forward through the evaluation map $\ev_t$ yields that $\rho_t = \Law(X^\rho_t)$ for all $t \in [0,T]$; indeed, by virtue of the associativity of the push-forward operation, we have that
\begin{equation*}
    \rho_t \coloneqq (\ev_t)_\sharp \rho = (\ev_t)_\sharp \Law(X^\rho) = (\ev_t)_\sharp((X^\rho)_\sharp\prob) = (\ev_t \circ X^\rho)_\sharp\prob = (X^\rho_t)_\sharp\prob = \Law(X^\rho_t).
\end{equation*}
Hence, the process $\barX \coloneqq X^\rho$ is a strong solution to the original McKean--Vlasov SDE \eqref{eq:McKean-Vlasov}. Pathwise uniqueness is obtained (in a particular class of strong solutions) through a Gr\"{o}nwall-type argument. For future reference, we state below the version of this classical result that will be used throughout the present work.

\begin{thm}\label{thm:SznitmanArgument} Let us assume that conditions \eqref{eq:GlobalLip} are satisfied, and let $p \ge 2\vee\barp$. Given a $d$-dimensional standard Brownian motion $\barB$ defined on a filtered probability space $(\Omega, \sF, (\sF_t)_{t\in[0,T]},\prob)$, and an $\sF_0$-measurable random variable $\barX_0\colon\Omega\to\real^d$ belonging to $L^p(\Omega,\sF, \prob)$, there exists a strong solution $\barX \colon \Omega \to C\bigl([0,T],\real^d\bigr)$ to problem \eqref{eq:McKean-Vlasov} with initial state $\barX_0$.
Moreover, the solution $\barX$ satisfies
\begin{equation}\label{eq:SznitmanFiniteMoments}
    \E\biggl[\sup_{t \in [0,T]}\norm{\barX_t}^p\biggr] < \infty,
\end{equation}
and it is pathwise unique in the class of strong solutions such that \eqref{eq:SznitmanFiniteMoments} holds.
\end{thm}

\begin{proof} The original proof of Sznitman \cite[Theorem 1.1, p.~172]{Sznitman} covers the case where $v$ is bounded and linear in the measure argument and $\sigma \equiv I_d$. The case $\barp = 2$ is treated in \cite[Theorem 2.2]{MeleardSDE} (see also \cite[Theorem 1.7]{CarmonaSDE} and \cite[Proposition 1, Section 2.2]{ChaosReviewI}). For the reader's convenience, we present a proof in the case $\barp \ge 1$ in Appendix~\ref{app:SzinitmanProof}.
\hfill
\end{proof}

\subsection{Lack of global Lipschitz continuity: the established proof for the mean-field CBO equation.}\label{subsec:EstablishedApproach} One of the main technical difficulties in the study of \eqref{eq:mean_field_CBO} is the lack of global Lipschitz continuity of the fields. Indeed, even assuming reasonable continuity and growth conditions on the objective function $f$, the CBO fields 
\begin{equation}\label{eq:fieldsCBO}
    v_\mu(x) = -\lambda\bigl(x - \cM_\beta(\mu)\bigr), \qquad \sigma_\mu(x) = S\bigl(x - \cM_\beta(\mu)\bigr)
\end{equation}
satisfy the Lipschitz condition only in a \textit{local} sense in the measure argument. More precisely, under suitable structural assumptions on $f$, one can show that there exists $\barp = p_\cM(s,\ell) \ge 1$ (the nature of the exponent $p_\cM(s,\ell)$ shall be clarified in Section~\ref{sec:StructuralAssumptions}), such that, for all $p\ge\barp$ and $R > 0$, there exist constants $L_{v,R}\,,L_{\sigma,R}>0$ satisfying
\begin{subequations}\label{eq:LocLipFieldsCBO}
\begin{align}
    \norm{v_{\mu_1}(x_1) - v_{\mu_2}(x_2)} &\le L_{v,R}\bigl(\norm{x_1 - x_2} + W_p(\mu_1, \mu_2) \bigr), \label{eq:LocLipDriftCBO} \\
    \norm{\sigma_{\mu_1}(x_1) - \sigma_{\mu_2}(zx_2)}_F &\le L_{\sigma,R}\bigl(\norm{x_1 - x_2} + W_p(\mu_1, \mu_2) \bigr), \label{eq:LocLipDiffCBO}
\end{align}
\end{subequations}
for all $x_1, x_2 \in \real^d$ and $\mu_1, \mu_2 \in \cP_{p,R}(\real^d)$ (the set $\cP_{p,R}(\real^d)$ has been defined in \eqref{eq:P_pR}). Inequalities \eqref{eq:LocLipFieldsCBO} are an immediate corollary of Proposition~\ref{prop:PropConsensusPointHoff}, which is a result of \cite{MeanFieldHoff}.

This obstruction to the application of the classical Sznitman's argument of well-posedness was already identified and addressed in \cite{CBO-Carrillo}. There, the authors devise a proof based on a fixed-point argument on the space $C([0,T],\real^d)$. The first step involves the formulation of an auxiliary equation, where the trajectory of the instantaneous consensus point in \eqref{eq:mean_field_CBO} is replaced by a fixed continuous curve $\gamma \in C([0,T],\real^d)$, namely
\begin{equation}\label{eq:AuxProblemCarrillo}
    \de X^\gamma_t = -\lambda\bigl(X^\gamma_t - \gamma_t\bigr)\, \de t + S\bigl(X^\gamma_t- \gamma_t\bigr)\,\de \barB_t, \qquad\text{ for } t \in [0,T];
\end{equation}
then, the fixed points of the map $ \cT \colon C([0,T],\real^d) \to C([0,T],\real^d)$, defined as
\begin{equation}\label{eq:DefT}
        \gamma \mapsto \cT(\gamma) \coloneqq \{t \mapsto \cM_\beta(\Law(X^\gamma_t))\},
\end{equation}
are sought for. The fixed-point problem for $\cT$ is handled via the Leray--Schauder fixed-point theorem \cite[Theorem 11.3]{GilbargTrudinger};
indeed, it can be shown that $\cT$ is compact, and that the set of $\gamma$ satisfying $\gamma = \tau \cT(\gamma)$ for some $\tau \in [0,1]$ is bounded. This argument proves the existence of a strong solution to \eqref{eq:mean_field_CBO}. Finally, through a Gr\"{o}nwall-type argument, pathwise uniqueness for problem \eqref{eq:mean_field_CBO} is obtained in the class of strong solutions such that the map $t \mapsto \cM_\beta(\rho_t)$ is continuous over $[0,T]$.

For the details, we refer the reader to the well-posedness results in \cite{CBO-Carrillo} and to their generalization in \cite[Theorem 2.3]{MeanFieldHoff}, which is, to the best of our knowledge, the most comprehensive account of the well-posedness results for the mean-field CBO equation.

We point out that this proof strategy is now well established in the CBO literature, and variants of it appear in several related works; see, for instance, \cite{FedCBOCarrillo,CBOFornasierHypersurfaces,CBOFornasierMultMin, CBOHuangMinMax,CBOHuangMultiplayer}.

\subsection{An alternative approach for the mean-field CBO equation: a proof \textit{à la} Sznitman through a truncation argument.}\label{subsec:AlternativeCBO}
In this paragraph, we outline the well-posedness argument for equation \eqref{eq:mean_field_CBO} developed in the present work. The main idea is to recover the global Lipschitz continuity of the fields by means of a truncation procedure, thereby bringing the problem back within the scope of Sznitman’s framework described in Section~\ref{subsec:ClassicalSznitman}. 

Since, for fixed $p \ge \barp = p_\cM(s,\ell)$ (see \eqref{eq:p_M} below) and $R > 0$, the local Lipschitz continuity conditions \eqref{eq:LocLipFieldsCBO} hold on the set $\real^d \times \cP_{p,R}(\real^d)$, we may localize the problem through a truncation function $\varphi_R\colon\cP_{p,R}(\real^d) \to [0,1]$ satisfying 
\begin{equation}\label{eq:RequirementTruncation}
\varphi_R(\mu) =
\begin{dcases}
1, &\gm_{p}(\mu) \le R, \\
0, &\gm_{p}(\mu) \ge R+1,
\end{dcases}
\end{equation}
and formulate the \emph{$R$-truncated problem} as 
\begin{equation}\label{eq:R_truncated_CBO}
    \begin{dcases}
        \de\barX_t^R = -\lambda\bigl(\barX^R_t - \varphi_R(\rho^R_t)\,\cM_\beta(\rho^R_t)\bigr)\,\de t + S\bigl(\barX^R_t - \varphi_R(\rho^R_t)\,\cM_\beta(\rho^R_t)\bigr)\,\de \barB_t\,, \\
        \rho^R_t = \Law(\barX^R_t),
    \end{dcases}
\end{equation}
for $t \in [0,T]$, with the prescription that $\barX^R_0 = \barX_0$. The fields driving the McKean--Vlasov SDE \eqref{eq:R_truncated_CBO} will be shown to be globally Lipschitz on $\real^d \times \cP_p(\real^d)$; therefore, under the assumption that the initial datum $\barX_0$ satisfies $\Law(\barX_0) \in \cP_p(\real^d)$, with $p \ge 2\vee \barp$, the argument of Sznitman (see Section~\ref{subsec:ClassicalSznitman} and Appendix~\ref{app:SzinitmanProof}) yields the existence of a strong solution to the $R$-truncated problem, which also coincides with a strong solution to \eqref{eq:mean_field_CBO} as long as $\varphi_R(\rho^R_t) = 1$. Finally, we derive an estimate on the growth of $\gm_p(\rho^R_t)$ that is uniform in $R$, which allows us to conclude that, if $R$ is taken large enough, we can extend the time interval on which $\varphi_R(\rho^R_t) = 1$ to the whole $[0,T]$. In this way, we obtain a strong solution to the original problem \eqref{eq:mean_field_CBO}.

Pathwise uniqueness is proved as a consequence of an \emph{a priori} estimate on the $p$-th moment of solutions to \eqref{eq:mean_field_CBO}, together with the uniqueness result provided by Sznitman's argument. This strategy allows us to obtain pathwise uniqueness in the class of strong solutions to \eqref{eq:mean_field_CBO} such that the map $t \mapsto \cM_\beta(\rho_t)$ is only bounded over $[0,T]$. 

The form of the $R$-truncated problem \eqref{eq:R_truncated_CBO} can be heuristically justified by observing that, when the $p$-th moment grows too large, the consensus point is replaced by the origin of $\real^d$, thereby inducing a contraction effect on the distribution of the representative particle and preventing the violation of the local Lipschitz conditions \eqref{eq:LocLipFieldsCBO}, which we recall to be valid, for all fixed $R>0$, in the subset $\cP_{p,R}(\real^d)$ of probability measures $\mu$ such that the $p$-th moment satisfies $\gm_p(\mu) \le R$. 

We remark that a similar approach has been used by the author in the recent work \cite{LavoroLeccese} to study the well-posedness of a CBO-inspired mean-field model coupled with dynamics in the convex set $[0,1]$, where this additional variable models the knowledge of the environment. Here, we expand the ideas of the aforementioned work to treat the original mean-field CBO equation. 

\section{Structural assumptions and their consequences}\label{sec:StructuralAssumptions}
In this section, we state the standing assumptions on the objects defining the mean-field CBO equation \eqref{eq:mean_field_CBO} and their consequences on the fields driving the dynamics. We start by specifying the class of objective functions $f$ considered in this work.

\begin{defn}\label{def:ClassO_sl}
Let $s,\ell \ge 0$ be fixed. We say that a function $f \colon \real^d \to \real$ belongs to the class $\cO(s,\ell)$ if
\begin{subequations}\label{eq:O}
\begin{enumerate}
    \item[(1)] there exists a constant $L_f \ge 0$ such that 
    \begin{equation}\label{eq:O1}
        \abs{f(x) - f(y)} \le L_f(1+\norm{x} + \norm{y})^s \norm{x-y}, \quad \text{for all } x,y \in \real^d;
    \end{equation}
    \item[(2)] $f$ is bounded from below by $f_* \coloneqq \inf_{x \in \real^d} f(x)$ and there exist constants $c_b$, $c_a$, $C_b$, $C_a > 0$ such that 
    \begin{equation}\label{eq:O2}
        c_b \norm{x}^\ell - C_b \le f(x) - f_* \le c_a \norm{x}^\ell + C_a, \quad \text{for all } x \in \real^d.
    \end{equation}
\end{enumerate}
\end{subequations}
\end{defn}
\noindent Conditions of this form are standard in the literature on CBO-type algorithms (see, e.g., \cite{CBO-Carrillo, MeanFieldHoff}). We also remark that the class $\cO(s,\ell)$ coincides with the class $\cA(s,\ell,\ell)$ introduced in~\cite{MeanFieldHoff}. The collection $\cO(s,\ell)$ is empty unless $\ell \le s+1$, and in the case $\ell = 0$, condition \eqref{eq:O2} reduces to a boundedness assumption on $f$. The continuity and growth conditions \eqref{eq:O} on the objective function play a central role in establishing the stability and growth estimates on the fields in \eqref{eq:mean_field_CBO}. Before stating the relevant results, we introduce the critical exponent
\begin{equation}\label{eq:p_M}
    p_\cM(s,\ell) \coloneqq 
    \begin{dcases}
        s + 2, & \text{ if } \ell = 0, \\
        1, & \text{ if } \ell > 0.
    \end{dcases}
\end{equation}
We are now in a position to recall the fundamental stability and sublinearity estimates regarding the instantaneous consensus point $\cM_\beta$.

\begin{prop}[{\cite[Proposition 3.1, Proposition A.4]{MeanFieldHoff}}] \label{prop:PropConsensusPointHoff} Let $f \in \cO(s,\ell)$, with $s,\ell\ge0$, and let $p \ge p_\cM(s,\ell)$. Then, for all $R > 0$, there exists $L_{\cM,R} > 0$, depending on $R$ and $p$, such that
\begin{equation}\label{eq:LocLipConsensusPoint}
    \norm{\cM_\beta(\mu_1) - \cM_\beta(\mu_2)} \le L_{\cM,R}\, W_p(\mu_1,\mu_2),\quad \text{for all }\mu_1,\mu_2 \in \cP_{p,R}(\real^d).
\end{equation}
Moreover, there exists a constant $C_\cM > 0$, depending on $p,\ell, c_b, C_b, c_a, C_a$, such that
\begin{equation}\label{eq:SubConsensusPoint}
    \norm{\cM_\beta(\mu)} \le C_\cM \, \gm_p(\mu),\quad\text{for all } \mu \in \cP_p(\real^d).
\end{equation}
\end{prop}

\begin{rem}\label{rem:S_Lip} We recall that, for the sake of generality, we assume the map $S \colon \real^d \to \real^{d\times d}$ in \eqref{eq:mean_field_CBO} to be globally Lipschitz continuous, and to satisfy $S(0) = 0_{d\times d}$, where $0_{d\times d}$ is the null matrix in $\real^{d\times d}$. We henceforth denote by $L_S$ its Lipschitz constant. Notice that the these two conditions imply the estimate
\begin{equation*}
    \norm{S(x)}_F \le L_S \norm{x},\quad \text{ for all } x\in\real^d.
\end{equation*}
\end{rem}

Estimates \eqref{eq:LocLipConsensusPoint} and \eqref{eq:SubConsensusPoint} are the key tool to implement our alternative technique of well-posedness of the mean-field CBO equation.

\section{Proof of existence through a truncation argument}\label{sec:Existence}
In this section, we present in detail the truncation argument outlined in Section~\ref{subsec:AlternativeCBO}, which will provide the existence of a strong solution to equation \eqref{eq:mean_field_CBO}.

We begin by constructing a \emph{cut-off} function defined on the space $\cP_p(\real^d)$ and satisfying condition~\eqref{eq:RequirementTruncation}. For all fixed $R>0$, we consider a cut-off function $\eta_R$ belonging to $C^\infty_c([0,\infty))$ (the space of infinitely differentiable functions $[0,\infty) \to \real$ with compact support), such that $0 \le \eta_R \le 1$ and
\begin{equation}
    \eta_R(z) =
    \begin{dcases}
        1, & z \le R, \\
        0, & z \ge R+1. \\
    \end{dcases}
\end{equation}
Whence, we define the truncation function $\varphi_R \colon \cP_p(\real^d) \to [0,1]$ as 
\begin{equation}\label{eq:defTruncFunc}
    \varphi_R(\mu) \coloneqq \eta_R\bigl(\gm_p(\mu)\bigr),\qquad \text{for all } \mu \in \cP_p(\real^d).
\end{equation}
The function $\varphi_R$ clearly obeys \eqref{eq:RequirementTruncation} and satisfies the following global Lipschitz property.

\begin{lemma}\label{lem:varphiLip} For all fixed $R>0$ and $p \ge 1$, the truncation function $\varphi_R \colon \cP_p(\real^d) \to [0,1]$ defined in \eqref{eq:defTruncFunc} is Lipschitz continuous on the Wasserstein space $(\cP_p(\real^d), W_p)$. More explicitly, there exists a constant $L_{\varphi,R} > 0$ such that for all $\mu_1, \mu_2 \in \cP_p(\real^d)$
\begin{equation}\label{eq:varphiLip}
    \abs{\varphi_R(\mu_1) - \varphi_R(\mu_2)} \le L_{\varphi,R}\,W_p(\mu_1,\mu_2).
\end{equation}
\end{lemma}

\begin{proof} Since $\eta_R \in C^\infty_c([0,\infty))$ is Lipschitz continuous by construction, it is enough to prove that the function $\gm_p \colon \cP_p(\real) \to [0,\infty)$ is Lipschitz continuous on $(\cP_p(\real^d), W_p)$. To this end, we observe that for every $\mu \in \cP_p(\real^d)$ we have
\begin{equation}\label{eq:m_equals_W}
    \gm_p(\mu) = W_p(\mu, \delta_0),
\end{equation}
where $\delta_0$ denotes the Dirac measure centered at the origin. Equation \eqref{eq:m_equals_W} is a consequence of the fact that the only admissible transport plan between $\mu$ and $\delta_0$ is $\mu \otimes \delta_0$, namely, $\Gamma(\mu,\delta_0) = \{\mu \otimes \delta_0\}$, in the notations of Section~\ref{subsec:Notation}. To prove that this is the case, it is enough to show that, if $\pi \in \Gamma(\mu,\delta_0)$, then $\pi = \mu \otimes \delta_0$; it is well-known that a sufficient condition for the last equality to hold is
\begin{equation}\label{eq:coincidence}
    \pi(A\times B) = \mu(A)\delta_0(B) = 
    \begin{dcases}
    0, & 0 \notin B, \\
    \mu(A), & 0 \in B, \\
    \end{dcases}
\end{equation}
for all Borel subsets $A,B \subseteq \real^d$. To prove \eqref{eq:coincidence}, we distinguish between two cases: if $0\notin B$, by monotonicity of $\pi$ and the transport constrain $\pi(\real^d\times B) = \delta_0(B)$, we have that \begin{equation*}
    0 \le \pi(A \times B) \le \pi(\real^d \times B) = \delta_0(B) = 0; 
\end{equation*}    
whereas, if $0\in B$, by additivity of $\pi$, the constrain $\pi(A \times \real^d) = \mu(A)$, and the relation $\pi(A\times (B\setminus \{0\})) = 0 = \pi(A\times (\real^d \setminus \{0\}))$ (descending from the previous case), we have
\begin{equation*}
    \pi(A \times B) = \pi(A \times \{0\}) + \pi(A\times (B\setminus \{0\})) = \pi(A \times \real^d) = \mu(A).
\end{equation*}
This proves \eqref{eq:coincidence}, whence $\Gamma(\mu,\delta_0) = \{\mu \otimes \delta_0\}$ follows.

Therefore, by combining \eqref{eq:m_equals_W} with the triangle inequality for $W_p$, we obtain that, for all $\mu_1,\mu_2 \in \cP_p(\real^d)$,
    \begin{equation*}
        \abs{\gm_p(\mu_1) - \gm_p(\mu_2)}
        = \abs{W_p(\mu_1,\delta_0) - W_p(\mu_2,\delta_0)}
        \le W_p(\mu_1,\mu_2),
    \end{equation*}
which proves that $\gm_p$ is $1$-Lipschitz on $(\cP_p(\real^d), W_p)$. The claim follows.
\hfill    
\end{proof}

We are now in a position to investigate the well-posedness of the $R$-truncated problem \eqref{eq:R_truncated_CBO}.
We observe that,  for any $R > 0$, the McKean--Vlasov SDE \eqref{eq:R_truncated_CBO} is driven by the $R$-truncated drift and diffusion fields $\tv^R \colon \real^d \times \cP_p(\real^d) \to \real^d$, $\tsigma^R \colon \real^d \times \cP_p(\real^d) \to \real^{d \times d}$
\begin{equation}\label{eq:R_truncated_fields}
    \tv^R_\mu(x) \coloneqq -\lambda \bigl(x - \varphi_R(\mu)\,\cM_\beta(\mu)\bigr), \qquad \tsigma^R_\mu(x) \coloneqq S \bigl(x - \varphi_R(\mu)\,\cM_\beta(\mu)\bigr).
\end{equation}
In the following proposition, we shall prove that the $R$-truncated fields defined in \eqref{eq:R_truncated_fields} are globally Lipschitz continuous.

\begin{prop}\label{prop:LipTruncFields} Let $f \in \cO(s,\ell)$, with $s,\ell\ge0$, and let $p \ge p_\cM(s,\ell)$. Then, for all fixed $R > 0$, there exist $\tL_{v,R}, \tL_{\sigma,R} > 0$ such that
\begin{subequations}\label{eq:LipTruncFieldsCBO}
\begin{align}
    \norm{\tv^R_{\mu_1}(x_1) - \tv^R_{\mu_2}(x_2)} &\le \tL_{v,R}\bigl(\norm{x_1 - x_2} + W_p(\mu_1, \mu_2) \bigr), \label{eq:LipTruncDriftCBO} \\
    \norm{\tsigma^R_{\mu_1}(x_1) - \tsigma^R_{\mu_2}(x_2)}_F &\le \tL_{\sigma,R}\bigl(\norm{x_1 - x_2} + W_p(\mu_1, \mu_2) \bigr), \label{eq:LipTruncDiffCBO}
\end{align}
\end{subequations}
for all $x_1, x_2 \in \real^d$ and $\mu_1, \mu_2 \in \cP_p(\real^d)$.
\end{prop}

\begin{proof} Let $R > 0$ be fixed. It is sufficient to prove that the $R$-truncated fields are Lipschitz with respect to each variable separately. Let us begin with the field $v$. We have that, for all fixed $\mu \in \cP_p(\real^d)$,
\begin{equation*}
    \norm{\tv^R_{\mu}(x_1) - \tv^R_{\mu}(x_2)} = \lambda \norm{x_1 - x_2},
\end{equation*}
for all $x_1, x_2 \in \real^d$. On the other hand, if we fix $x \in \real^d$, then we can write
\begin{equation}\label{eq:intermediateTruncLip}
\begin{split}
    &\norm{\tv^R_{\mu_1}(x)-\tv^R_{\mu_2}(x)} = \lambda \norm{\varphi_R(\mu_1)\,\cM_\beta(\mu_1) - \varphi_R(\mu_2)\,\cM_\beta(\mu_2)} \\
    \le&\,\lambda\bigl( \abs{\varphi_R(\mu_1)}\, \norm{\cM_\beta(\mu_1) - \cM_\beta(\mu_2)} + \norm{\cM_\beta(\mu_2)}\, \abs{\varphi_R(\mu_1) - \varphi_R(\mu_2)} \bigr)
\end{split}
\end{equation}
for all $\mu_1, \mu_2 \in \cP_p(\real^d)$. 
Now, we observe that the hypotheses of Proposition~\ref{prop:PropConsensusPointHoff} are satisfied and that three cases can occur. If both measures $\mu_1, \mu_2$ satisfy $\gm_p(\mu_1), \gm_p(\mu_2) \le R+1$, then by recalling Lemma \ref{lem:varphiLip}, the local Lipschitz property \eqref{eq:LocLipConsensusPoint}, the sublinearity property \eqref{eq:SubConsensusPoint}, and the fact that $0\le\varphi_R\le1$, estimate \eqref{eq:intermediateTruncLip} becomes
\begin{equation*}
\begin{split}
    \norm{\tv^R_{\mu_1}(x)-\tv^R_{\mu_2}(x)} \le&\, \lambda\bigl( L_{\cM,R+1}\,W_p(\mu_1, \mu_2) + C_\cM\,\gm_p(\mu_2) \,L_{\varphi,R}\,W_p(\mu_1,\mu_2)\bigr) \\
    \le&\, \lambda\bigl(L_{\cM,R+1} + C_\cM (R+1) L_{\varphi,R}\bigr)\,W_p(\mu_1,\mu_2).
\end{split}
\end{equation*}
If instead $\gm_p(\mu_1) > R+1$ and $\gm_p(\mu_2) \le R+1$ (the case $\gm_p(\mu_1) \le R+1$ and $\gm_p(\mu_2) > R+1$ being completely symmetric), we have that $\varphi_R(\mu_1) = 0$, and again by the sublinearity property \eqref{eq:SubConsensusPoint} and estimate \eqref{eq:intermediateTruncLip} we obtain
\begin{equation*}
    \norm{\tv^R_{\mu_1}(x)-\tv^R_{\mu_2}(x)} \le \lambda\,C_\cM (R+1) L_{\varphi,R} W_p(\mu_1,\mu_2).
\end{equation*}
Finally, when $\gm_p(\mu_1), \gm_p(\mu_2) > R+1$, the inequality trivializes since $\varphi_R(\mu_1) = \varphi_R(\mu_2) = 0$.
Therefore, estimate \eqref{eq:LipTruncDriftCBO} holds with $\tL_{v,R} \coloneqq \max\bigl\{\lambda\bigl(L_{\cM,R+1} + C_\cM (R+1) L_{\varphi,R}\bigr),\lambda\bigr\}$. Recalling the assumptions on $S$ of Remark~\ref{rem:S_Lip}, we observe that
\begin{equation*}
    \norm{\tsigma^R_{\mu_1}(x_1) - \tsigma^R_{\mu_2}(x_2)}_F \le L_S \lambda^{-1} \norm{\tv^R_{\mu_1}(x_1) - \tv^R_{\mu_2}(x_2)}.
\end{equation*}
Hence, we may take $\tL_{\sigma,R} \coloneqq L_S\lambda^{-1} \tL_{v,R}$ in  \eqref{eq:LipTruncDiffCBO}. The proof is concluded.
\hfill    
\end{proof}

By virtue of Proposition \ref{prop:LipTruncFields}, we can now apply Theorem~\ref{thm:SznitmanArgument} to obtain the well-posedness of the $R$-truncated problem \eqref{eq:R_truncated_CBO} for all $R>0$.

\begin{prop}\label{prop:WellPosednessRTrunc} 
Let $f \in \cO(s,\ell)$, with $s,\ell\ge0$, and $R > 0$ be fixed. Given a $d$-dimensional standard Brownian motion $\barB$ defined on a filtered probability space $(\Omega, \sF, (\sF_t)_{t\in[0,T]},\prob)$, and an $\sF_0$-measurable random variable $\barX_0\colon\Omega\to\real^d$ belonging to $L^p(\Omega,\sF, \prob)$, with $p \ge 2 \vee p_\cM(s,\ell)$, there exists a strong solution $\barX^R \colon \Omega \to C\bigl([0,T],\real^d\bigr)$ to problem \eqref{eq:R_truncated_CBO} with initial state $\barX_0$.
Moreover, the solution $\barX^R$ satisfies
\begin{equation}\label{eq:TruncSolFiniteMoments}
    \E\biggl[\sup_{t \in [0,T]}\norm{\barX^R_t}^p\biggr] < \infty,
\end{equation}
and it is pathwise unique in the class of strong solutions such that \eqref{eq:TruncSolFiniteMoments} holds.
\end{prop}

\begin{proof} It is a direct consequence of the global Lipschitz estimates \eqref{eq:LipTruncFieldsCBO} and of Sznitman's argument of well-posedness, as stated in Theorem~\ref{thm:SznitmanArgument}. \hfill
\end{proof}

We now consider the pathwise unique strong solution $\barX^R$ to \eqref{eq:R_truncated_CBO} constructed in Proposition~\ref{prop:WellPosednessRTrunc}. We observe that, as a consequence of the Dominated Convergence Theorem, condition \eqref{eq:TruncSolFiniteMoments} implies the continuity of the map $t \mapsto \gm_p(\rho^R_t)$ on $[0,T]$. More importantly, the strong solution $\barX^R$ is also a strong solution to the original equation \eqref{eq:mean_field_CBO} on the time interval $[0,T_R]$, where
\begin{equation}\label{eq:Time_T_R}
    T_R \coloneqq \inf\{t \in [0,T] \,\colon \gm_p(\rho^R_t) > R \},
\end{equation}
with the prescription that $T_R = T$ when this set is empty; indeed, $\varphi_R(\rho^R_t) \equiv 1$ on $[0,T_R]$. Furthermore, by pathwise uniqueness, if $R_1 < R_2$, then, $\prob$-almost surely, $\barX^{R_1} = \barX^{R_2}$ on $[0,T_{R_1}]$, whence $T_{R_1} \le T_{R_2}$. The following estimate, which is uniform in $R$, on the growth of $\gm_p(\rho^R_t)$ allows us to extend the time interval $[0,T_R]$ to the whole $[0,T]$, provided that $R$ is chosen large enough.

\begin{prop} Let us assume that the hypotheses of Proposition~\ref{prop:WellPosednessRTrunc} are satisfied. Moreover, let $\barX^R$ denote, for each $R > 0$, the solution to problem~\eqref{eq:R_truncated_CBO} constructed therein, and let $\rho_0$ be the law of the initial state $\barX_0$. Then there exists a constant $\cC_0 > 0$, depending only on $\lambda$, $\cC_\cM$, $L_S$, $p$, $d$, and $T$, such that
\begin{equation}\label{eq:UniformEstimateOnMoments}
    \sup_{t \in [0,T]} \gm_p(\rho^R_t) \le \cC_0 \gm_p(\rho_0), 
    \qquad \text{for all } R > 0.
\end{equation}
\end{prop}

\begin{proof} Since, by hypothesis, $p \ge 2 \vee p_\cM(s,\ell)$, we can apply It\={o}'s formula to the $C^2(\real^d)$ function $\psi(x) = \norm{x}^p$ and to the process $\barX^R$ for all fixed $R>0$ to obtain 
\begin{equation*}
    \frac{\de}{\de t}\gm_p^p(\rho^R_t) = \int_{\real^d} \langle \tv^R_{\rho^R_t}(x),\nabla\psi(x) \rangle\,\de\rho^R_t(x) + \frac{1}{2}\int_{\real^d}\trace\Bigl(\tsigma^R_{\rho^R_t}(x)^\top H_\psi(x)\tsigma^R_{\rho^R_t}(x)\Bigr)\,\de\rho^R_t(x)
\end{equation*}
(see \eqref{eq:momentum_def} for the definition of $\gm_p$); since $\nabla\psi(x) = p\norm{x}^{p-2} x$, we get  
\begin{equation*}
\begin{split}
    \frac{\de}{\de t}\gm_p^p(\rho^R_t) = \int_{\real^d}p\norm{x}^{p-2}&\langle \tv^R_{\rho^R_t}(x), x\rangle\,\de\rho^R_t(x) \\
    &+ \frac{1}{2}\int_{\real^d}\trace\Bigl(\tsigma^R_{\rho^R_t}(x)^\top H_\psi(x)\tsigma^R_{\rho^R_t}(x)\Bigr)\,\de\rho^R_t(x) \eqqcolon \mathrm{I}_t + \mathrm{II}_t.
\end{split}
\end{equation*}
By recalling the form \eqref{eq:R_truncated_fields} of the $R$-truncated fields, we can now estimate the two terms $\mathrm{I}_t$, $\mathrm{II}_t$. Let us begin with the quantity $\mathrm{I}_t$. By Cauchy-Schwarz inequality, triangle inequality, and the fact that the truncation function $\varphi_R$ satisfies $0\le\varphi_R\le1$, we have that
\begin{equation*}
\begin{split}
    \mathrm{I}_t &\le p\int_{\real^d}\norm{x}^{p-1}\norm{\tv^R_{\rho^R_t}(x)}\,\de\rho^R_t(x) \le \lambda p\int_{\real^d}\norm{x}^{p-1}\bigl(\norm{x} + \norm{\varphi_R(\rho^R_t)\cM_\beta(\rho^R_t)}\bigr)\,\de\rho^R_t(x) \\
    &\le \lambda p \biggl(\gm_p^p(\rho^R_t) + \int_{\real^d}\norm{x}^{p-1}\norm{\cM_\beta(\rho^R_t)}\,\de\rho^R_t(x)\biggr).
\end{split}
\end{equation*}
Let us now focus on the last integral. By H\"{o}lder inequality applied with exponents $p$ and $\frac{p}{p-1}$, and growth estimate \eqref{eq:SubConsensusPoint}, we have
\begin{equation*}
\begin{split}
&\int_{\real^d}\norm{x}^{p-1}\norm{\cM_\beta(\rho^R_t)}\,\de\rho^R_t(x)  \le \biggl(\int_{\real^d} \norm{x}^p\,\de\rho^R_t(x)\biggr)^{\frac{p-1}{p}}\biggl(\int_{\real^d} \norm{\cM_\beta(\rho^R_t)}^p\,\de\rho^R_t(x)\biggr)^\frac{1}{p} \\
&= \bigl(\gm_p(\rho^R_t)\bigr)^{p-1}\cdot\norm{\cM_\beta(\rho^R_t)} \le \bigl(\gm_p(\rho^R_t)\bigr)^{p-1}\cdot \cC_\cM\,\gm_p(\rho^R_t) = \cC_\cM\,\gm_p^p(\rho^R_t).
\end{split}
\end{equation*}
Hence, we can write that for all $t \in [0,T]$
\begin{equation*}
    \mathrm{I}_t \le \lambda p \bigl( 1 + \cC_\cM\bigr) \gm_p^p(\rho^R_t).
\end{equation*}
Let us now address the term $\mathrm{II}_t$. Since, under our hypotheses, $\psi$ is of class $C^2(\real^d)$ and convex, its Hessian matrix $H_\psi(x)$ is symmetric and positive semidefinite at every point $x \in \real^d$; therefore, its matrix square root is well-defined and we can write 
\begin{equation*}
    H_\psi(x) = \Bigl(\sqrt{H_\psi(x)}\Bigr)^{\!\top} \sqrt{H_\psi(x)}, \qquad \text{for all } x \in \real^d.
\end{equation*}
Whence, by recalling the representation formula for the Frobenius norm $\norm{A}_F^2 = \trace(A^\top A)$, and using its submultiplicative property, namely $\norm{AB}_F \le \norm{A}_F\norm{B}_F$, 
for all $A,B \in \real^{d\times d}$, we can estimate the integrand defining $\mathrm{II}_t$ as
\begin{equation*}
\begin{split}
    &\,\trace\Bigl(\tsigma^R_{\rho^R_t}(x)^\top H_\psi(x)\tsigma^R_{\rho^R_t}(x)\Bigr) = \trace\biggl(\Bigl(\sqrt{H_\psi(x)}\,\tsigma^R_{\rho^R_t}(x) \Bigr)^{\!\top}\sqrt{H_\psi(x)}\,\tsigma^R_{\rho^R_t}(x)\biggr) \\
    =&\, \Bigl\lVert\sqrt{H_\psi(x)}\,\tsigma^R_{\rho^R_t}(x)\Bigr\rVert_F^2 \le \Bigl\lVert\sqrt{H_\psi(x)}\Bigr\rVert_F^2 \Bigl\lVert\tsigma^R_{\rho^R_t}(x)\Bigr\rVert_F^2 = \trace H_\psi(x) \cdot \Bigl\lVert\tsigma^R_{\rho^R_t}(x)\Bigr\rVert_F^2 \\
    =&\, \Delta\psi(x) \cdot \Bigl\lVert\tsigma^R_{\rho^R_t}(x)\Bigr\rVert_F^2 = p(p - 2 + d)\norm{x}^{p-2}\cdot \Bigl\lVert\tsigma^R_{\rho^R_t}(x)\Bigr\rVert_F^2\,,
\end{split}
\end{equation*}
where we have made explicit the Laplacian $\Delta\psi(x) = p(p- 2 + d)\norm{x}^{p-2}$ in the last equality.
We now let $K \coloneqq p(p- 2 + d)/2$. As a consequence of the previous estimate and of the assumed properties of $S$ (see Remark~\ref{rem:S_Lip}), we deduce that
\begin{equation*}
\begin{split}
    \mathrm{II}_t &\le K \int_{\real^d} \norm{x}^{p-2} \Bigl\lVert\tsigma^R_{\rho^R_t}(x)\Bigr\rVert_F^2\,\de\rho^R_t(x) \le K L_S^2 \int_{\real^d} \norm{x}^{p-2}\bigl(\norm{x} + \norm{\cM_\beta(\rho^R_t)}\bigr)^2 \,\de\rho^R_t(x) \\
    &\le 2KL_S^2 \biggl(\gm_p^p(\rho^R_t) + \int_{\real^d}\norm{x}^{p-2}\norm{\cM_\beta(\rho^R_t)}^2\,\de\rho^R_t(x)\biggr).
\end{split}
\end{equation*}
Now, by using H\"{o}lder inequality with exponents $\frac{p}{2}$ and $\frac{p}{p-2}$ and property \eqref{eq:SubConsensusPoint}, we can estimate the last integral as follows
\begin{equation*}
\begin{split}
&\int_{\real^d}\norm{x}^{p-2}\norm{\cM_\beta(\rho^R_t)}^2\,\de\rho^R_t(x)  \le \biggl(\int_{\real^d} \norm{x}^p\,\de\rho^R_t(x)\biggr)^\frac{p-2}{p}\biggl(\int_{\real^d} \norm{\cM_\beta(\rho^R_t)}^p\,\de\rho^R_t(x)\biggr)^{\frac{2}{p}} \\
&= \bigl(\gm_p(\rho^R_t)\bigr)^{p-2}\cdot\norm{\cM_\beta(\rho^R_t)}^2 \le \bigl(\gm_p(\rho^R_t)\bigr)^{p-2}\cdot \cC_\cM^2\,\gm_p^2(\rho^R_t) = \cC_\cM^2\,\gm_p^p(\rho^R_t).
\end{split}
\end{equation*}
Therefore, for all $t \in [0,T]$, we have
\begin{equation*}
    \mathrm{II}_t \le p(p - 2 + d)L_S^2\bigl(1 + \cC_\cM^2 \bigr)\gm_p^p(\rho^R_t).
\end{equation*}
Hence, letting $\cC_1 \coloneqq \lambda p (1+\cC_\cM) + p(p - 2 + d)L_S^2(1 + \cC_\cM^2)$, we finally obtain that
\begin{equation*}
    \frac{\de}{\de t}\gm_p^p(\rho^R_t) \le \cC_1\, \gm_p^p(\rho^R_t), \qquad \text{for all }t \in [0,T].
\end{equation*}
Since, by virtue of \eqref{eq:TruncSolFiniteMoments}, the function $t \mapsto \gm_p^p(\rho^R_t) = \E\bigl[\norm{\barX^R_t}^p\bigr]$ is bounded over $[0,T]$, we can apply Gr\"{o}nwall inequality to infer that
\begin{equation*}
    \gm_p^p(\rho^R_t) \le e^{\cC_1 t}\,\gm_p^p(\rho_0), \qquad \text{for all }t \in [0,T].
\end{equation*}
By taking the $p$-th root of both sides and the supremum over $[0,T]$, inequality \eqref{eq:UniformEstimateOnMoments} follows with $\cC_0 \coloneqq (e^{\cC_1 T})^{1/p}$.
\hfill    
\end{proof}

By virtue of the previous estimate \eqref{eq:UniformEstimateOnMoments}, it is enough to consider $\barR > 0$ such that  
\begin{equation*}\label{eq:ChoiceOfBarR}
    \sup_{t \in [0,T]} \gm_p(\rho^{\barR}_t) \le \cC_0 \gm_p(\rho_0) \le \barR
\end{equation*}
to deduce that $T_{\barR} = T$, so that the process $\barX \coloneqq \barX^{\barR}$ is a strong solution to the mean-field CBO equation \eqref{eq:mean_field_CBO} on the whole interval $[0,T]$. The existence part in Theorem~\ref{thm:MainResult} is completely proved.

\section{An \emph{a priori} estimate and the proof of uniqueness}\label{sec:Uniqueness}
In this section we address the problem of pathwise uniqueness of the solution to \eqref{eq:mean_field_CBO}. The result will be obtained by combining the uniqueness result for the $R$-truncated problem \eqref{eq:R_truncated_CBO} with an \emph{a priori} estimate on the $p$-th moments of the solutions to \eqref{eq:mean_field_CBO}.

\begin{prop}\label{prop:APrioriEstimateCBO} Let $f \in \cO(s,\ell)$, with $s,\ell\ge0$, and let $\barX_0 \colon \Omega \to \real^d$ be an $\sF_0$-measurable random variable belonging to $L^p(\Omega,\sF,\prob)$, with $p \ge 2\vee p_\cM(s,\ell)$. Then any strong solution to problem \eqref{eq:mean_field_CBO} with initial state $\barX_0$, and such that the map $t \mapsto \cM_\beta(\rho_t)$ is bounded over $[0,T]$, satisfies the following \emph{a priori} estimate
\begin{equation}\label{eq:APrioriEstimateCBO}
\E\biggl[\sup_{t\in[0,T]} \norm{\barX_t}^p\biggr] \le C_\rho\Bigl(1 + \E\bigl[\norm{\barX_0}^p\bigr]\Bigr),
\end{equation}
where the constant $C_\rho > 0$ depends only on $\lambda$, $L_S$, $p$, $T$, and $\sup_{t\in[0,T]}\norm{\cM_\beta(\rho_t)}$.
\end{prop}

\begin{proof} 
Since by assumption the map $t \mapsto \cM_\beta(\rho_t)$ is bounded over $[0,T]$, the fields~$v$ and~$\sigma$ (see \eqref{eq:fieldsCBO}) driving the McKean--Vlasov equation \eqref{eq:mean_field_CBO} satisfy the following sublinearity estimates, uniformly in time:
\begin{equation}\label{eq:SubEstimatesFieldsCBO}
    \norm{v_{\rho_t}(x)} \le M_v(1 + \norm{x}), \quad \norm{\sigma_{\rho_t}(x)}_F \le M_\sigma(1 + \norm{x}), \quad \text{for all } (x,t) \in \real^d \times [0,T], 
\end{equation}
with $M_v \coloneqq \lambda \max\{1,M_{\rho}\}$ and $M_\sigma \coloneqq L_S \max\{1,M_{\rho}\} $, where $M_{\rho} \coloneqq \sup_{t \in [0,T]}\norm{\cM_\beta(\rho_t)}$. Indeed, we readily have that, for all $(x,t) \in \real^d \times [0,T]$
\begin{align*}
    \norm{v_{\rho_t}(x)} &\le \lambda(\norm{\cM_\beta(\rho_t)} + \norm{x}) \le \lambda (M_\rho + \norm{x}), \\
    \norm{\sigma_{\rho_t}(x)}_F &\le L_S(\norm{\cM_\beta(\rho_t)} + \norm{x}) \le L_S(M_\rho + \norm{x}),
\end{align*}
 and the claim follows (we observe that the considerations of Remark~\ref{rem:S_Lip} have been exploited in the first inequality in the estimate for $\sigma$). Therefore, by virtue of the sublinearity estimates \eqref{eq:SubEstimatesFieldsCBO}, we can invoke the result of \cite[Theorem 9.1]{Stochastic-Calculus} to deduce the \emph{a priori} estimate \eqref{eq:APrioriEstimateCBO}, with $C_\rho$ depending only on $\lambda$, $L_S$, $p$, $T$, and $M_\rho$.
\hfill    
\end{proof}

Finally, we are in a position to prove the pathwise uniqueness result of Theorem~\ref{thm:MainResult}.

\begin{cor}\label{cor:PathwiseUniquenessCBO} Let us assume that the hypotheses of Proposition \ref{prop:APrioriEstimateCBO} are satisfied. Pathwise uniqueness for problem \eqref{eq:mean_field_CBO} holds in the class of strong solutions such that the map $t \mapsto \cM_\beta(\rho_t)$ is bounded over $[0,T]$.
\end{cor}

\begin{proof} Let us consider two strong solutions $\barX^1$, $\barX^2$ to problem \eqref{eq:mean_field_CBO}, relative to the same initial datum $\barX_0 \in L^p(\Omega,\sF,\prob)$ (with $p \ge 2\vee p_\cM(s,\ell)$) and to the same Brownian motion $\barB$. Moreover, denoting by $\rho^1_t$, $\rho^2_t$ their respective law at time $t$, we assume that the two maps $t \mapsto \cM_\beta(\rho^i_t)$, $i = 1,2$, are both bounded over $[0,T]$.
We observe that, as a consequence of Proposition~\ref{prop:APrioriEstimateCBO}, we can choose $\barR > 0$ large enough so that
\begin{equation*}
\sup_{t\in[0,T]}\gm_p(\rho^i_t) \le \E\biggl[\sup_{t\in[0,T]} \norm{\barX^i_t}^p\biggr]^{1/p} \le \biggl[C_{\rho^i}\Bigl(1 + \E\bigl[\norm{\barX_0}^p\bigr]\Bigr)\biggr]^{1/p} < \barR,
\end{equation*}
for $i = 1,2$. Hence, the processes $\barX^1$, $\barX^2$ can also be  regarded as two solutions to the $\barR$-truncated problem \eqref{eq:R_truncated_CBO}. Indeed, the truncation function $\varphi_{\barR}$ is always equal to $1$ along the law of these two solutions, and the truncated problem \eqref{eq:R_truncated_CBO} coincides with the original CBO \eqref{eq:mean_field_CBO}. Since $\barX^1$, $\barX^2$ necessarily satisfy condition \eqref{eq:TruncSolFiniteMoments}, the pathwise uniqueness result for the $\barR$-truncated problem (see Proposition \ref{prop:WellPosednessRTrunc}) implies the desired pathwise uniqueness result for the mean-field CBO equation \eqref{eq:mean_field_CBO}.
\hfill    
\end{proof}

\appendix
\section{Sznitman's argument of well-posedness}\label{app:SzinitmanProof}
We begin this appendix by recalling a useful $L^p$ estimate for It\={o} stochastic integrals (see, e.g., \cite[Proposition 8.4]{Stochastic-Calculus} and \cite[page 116]{SV}).

\begin{thm}[Burkholder--Davis--Gundy inequality] 
Let $(B_t)_{t\in[0,T]}$ be a standard $\real^m$-valued Brownian motion and let $(G_t)_{t\in[0,T]}$ be a progressively measurable, $\real^{d\times m}$-valued stochastic process, both defined on a common filtered probability space $\bigl(\Omega, \sF, (\mathscr{F}_t)_{t\in[0,T]},\prob\bigr)$.
Let $p \ge 2$ and suppose that $\int_0^T\,\norm{G_t}_{F}^p\,\de t < +\infty$, $\prob$-almost surely, then, for all $t \in [0,T]$, the following inequality holds:
\begin{equation}\label{eq:BDG}
    \E\biggl[\sup_{u\in[0,t]}\norm*{\int_0^u G_s\,\de B_s}^p\biggr] \le c_p\, t^\frac{p-2}{2}\, \E\biggl[\int_0^t\norm{G_s}_{F}^p\,\de s\biggr],
\end{equation}
where $c_p = \bigl[\frac{1}{2}\, p^{p+1} (p-1)^{1-p}\bigr]^\frac{p}{2}$.
\end{thm}

By means of this classical result, we can illustrate in detail the formulation of Sznitman's argument stated in Theorem~\ref{thm:SznitmanArgument}.

\begin{proof}[Proof of Theorem~\ref{thm:SznitmanArgument}.]
We first observe that, since $p \ge 2 \vee \barp$, inequalities \eqref{eq:GlobalLip} still hold with $\barp$ replaced by $p$, by the properties of the Wasserstein distance. Next, we fix $\Psi \in \cP_p(C([0,T],\real^d))$ and formulate the auxiliary time-dependent SDE \eqref{eq:AuxProblem}. Since the measure on path space $\Psi$ has finite $p$-th moment, as a consequence of the Dominated Convergence Theorem one has that the map $t \mapsto \Psi_t \coloneqq (\ev_t)_\sharp\Psi \in \cP_p(\real^d)$ is continuous over $[0,T]$ with respect to the metric $W_p$. Whence, by \eqref{eq:GlobalLip}, the (time dependent) fields driving the auxiliary equation \eqref{eq:AuxProblem} satisfy the following global Lipschitz estimates
\begin{equation}\label{eq:AuxGlobalLip}
    \norm{v_{\Psi_t}(x_1) - v_{\Psi_t}(x_2)} \le L_v \norm{x_1 - x_2}, \qquad \norm{\sigma_{\Psi_t}(x_1) - \sigma_{\Psi_t}(x_2)}_F \le L_\sigma \norm{x_1 - x_2},
\end{equation}
for $x_1, x_2 \in \real^d$ and $t \in [0,T]$, as well as the sublinearity estimates
\begin{equation}\label{eq:AuxGlobalSub}
    \norm{v_{\Psi_t}(x)} \le M^\Psi_v\bigl(1 + \norm{x}\bigr), \qquad \norm{\sigma_{\Psi_t}(x)}_F \le M^\Psi_\sigma \bigl(1 + \norm{x}\bigr),
\end{equation}
for $x \in \real^d$ and $t\in [0,T]$, where $M^\Psi_v, M^\Psi_\sigma > 0$ depend  on $\sup_{t \in [0,T]}W_p(\Psi_0,\Psi_t) < +\infty$. Therefore, by virtue of the standard theory of well-posedness for SDEs (see, e.g., \cite[Chapter 9]{Stochastic-Calculus}), estimates \eqref{eq:AuxGlobalLip}, \eqref{eq:AuxGlobalSub}, and the integrability assumption on the initial datum $\barX_0$ imply that equation \eqref{eq:AuxProblem} admits a pathwise unique strong solution $X^\Psi \colon \Omega \to C([0,T],\real^d)$ whose law belongs to $\cP_p(C([0,T],\real^d))$, hence, the map $\cS$ in \eqref{eq:DefS} is well-defined.

Now we prove that $\cS$ admits a unique fixed point in the complete and sparable metric space $(\cP_p(C([0,T],\real^d)), W_p)$. To this end, we introduce, for all $t \in [0,T]$, and for all measures $\Psi^1, \Psi^2 \in \cP_p(C([0,T],\real^d))$, the following quantity
\begin{equation}\label{eq:W_p_t}
    W_{p,t}(\Psi^1,\Psi^2) \coloneqq \biggl(\inf_{\pi \in \Gamma(\Psi^1,\Psi^2)} \int_{C([0,T],\real^d)^2} \sup_{u\in[0,t]} \norm{\varphi_1(u) - \varphi_2(u)}^p \, \de \pi(\varphi_1,\varphi_2)\biggr)^{1/p}.
\end{equation}
We observe that the infimum in \eqref{eq:W_p_t} is always achieved, by the standard theory of Optimal Transport on complete and sparable metric spaces (see, e.g.,\cite[Theorem~2.10]{Optimal-Transport-Ambrosio}). Moreover, if $s \le t$, then $W_{p,s} \le W_{p,t}$, and clearly $W_{p,T}$ coincides with the metric $W_p$.

Next, let us fix $\Psi^1, \Psi^2 \in \cP_p(C([0,T],\real^d))$, and let us denote by $X^1$ and $X^2$ the solutions to \eqref{eq:AuxProblem} associated to $\Psi^1$ and $\Psi^2$, respectively. By applying the elementary estimate $\norm{a + b}^p \le 2^{p-1}(\norm{a}^p + \norm{b}^p)$, H\"{o}lder inequality, the Burkholder--Davis--Gundy inequality \eqref{eq:BDG}, and the global Lipschitz estimates \eqref{eq:GlobalLip} we obtain, for all $t \in [0,T]$,
\begin{equation}\label{eq:HoldBDGLip}
\begin{split}
    &\E\biggl[\sup_{u\in[0,t]} \norm{X^1_u - X^2_u}^p\biggr] \le\, (2t)^{p-1}\E\biggl[\int_0^t\norm{v_{\Psi^1_s}(X^1_s) - v_{\Psi^2_s}(X^2_s)}^p\,\de s \biggr] \\
    &\,+ 2^{p-1}c_p t^\frac{p-2}{2}\,\E\biggl[\int_0^t\norm{\sigma_{\Psi^1_s}(X^1_s) - \sigma_{\Psi^2_s}(X^2_s)}_F^p\,\de s\biggr] \\
    \le&\, 2^{p-1}\Bigl((2t)^{p-1}L_v^p + 2^{p-1}c_p t^\frac{p-2}{2}L_\sigma^p\Bigr)\cdot \E\biggl[\int_0^t \bigl(\norm{X^1_s - X^2_s}^p + W_p^p(\Psi^1_s,\Psi^2_s)\bigr)\,\de s\biggr].
\end{split}
\end{equation}
Now we estimate the last integrand above. Let us fix $s \in [0,t]$, and let us consider a measure $\tilde{\pi} \in \Gamma(\Psi^1,\Psi^2)$ which is optimal for $W_{p,s}$.
If we consider the probability measure over $\real^d \times \real^d$ defined as $\tilde{\pi}_s \coloneqq (\ev_s,\ev_s)_\sharp \tilde{\pi} \in \Gamma(\Psi^1_s,\Psi^2_s)$, then we have that
\begin{equation*}
\begin{split}
    W_p^p(\Psi^1_s,\Psi^2_s) &\le \int_{(\real^d)^2}\norm{x_1 - x_2}^p\,\de\tilde{\pi}_s(x_1,x_2) = \int_{C([0,T],\real^d)^2}\norm{\varphi_1(s) - \varphi_2(s)}^p\,\de\tilde{\pi}(\varphi_1,\varphi_2) \\
    &\le \int_{C([0,T],\real^d)^2} \sup_{u\in[0,s]} \norm{\varphi_1(u) - \varphi_2(u)}^p \, \de \tilde{\pi}(\varphi_1,\varphi_2) = W_{p,s}^p(\Psi^1,\Psi^2),
\end{split}
\end{equation*}
where the last equality follows from the optimality of the transport plan $\tilde{\pi}$. 

Hence, if we let $\cK_0 \coloneqq 2^{p-1}\bigl((2t)^{p-1}L_v^p + 2^{p-1}c_p t^\frac{p-2}{2}L_\sigma^p\bigr)$, we have that, for all $t \in [0,T]$,
\begin{equation*}
    \E\biggl[\sup_{u\in[0,t]} \norm{X^1_u - X^2_u}^p\biggr] \le \cK_0 \biggl(\int_0^t \E\biggl[\sup_{u\in[0,s]} \norm{X^1_u - X^2_u}^p \biggr] \,\de s + \int_0^t W_{p,s}^p(\Psi^1,\Psi^2)\,\de s\biggr),
\end{equation*}
therefore, by Gr\"{o}nwall inequality (whose validity is ensured by the boundedness of the map $t \mapsto \E\bigl[\sup_{u\in[0,t]} \norm{X^1_u - X^2_u}^p\bigr]$ over $[0,T]$, and by the fact that the map $t \mapsto \int_0^t W_{p,s}^p(\Psi^1,\Psi^2)\,\de s \le T W_p^p(\Psi^1,\Psi^2)$ is non-decreasing and bounded), we conclude that
\begin{equation*}
    \E\biggl[\sup_{u\in[0,t]} \norm{X^1_u - X^2_u}^p\biggr] \le e^{\cK_0 t}\cK_0 \int_0^t W_{p,s}^p(\Psi^1,\Psi^2)\,\de s.
\end{equation*}
Finally, since $\Law(X^1,X^2) = (X^1,X^2)_\sharp\prob \in \Gamma(\cS(\Psi^1),\cS(\Psi^2))$, we obtain the estimate
\begin{equation}\label{eq:bottigliarotta}
    W_{p,t}^p(\cS(\Psi^1),\cS(\Psi^2)) \le \E\biggl[\sup_{u\in[0,t]} \norm{X^1_u - X^2_u}^p\biggr] \le \cK \int_0^t W_{p,s}^p(\Psi^1,\Psi^2)\,\de s
\end{equation}
for all $t \in [0,T]$, with $\cK \coloneqq e^{\cK_0 T}\cK_0$. From \eqref{eq:bottigliarotta}, we readily obtain the following set of inequalities, which hold for all $t \in [0,T]$ and $n \in \nat_+$
\begin{subequations}\label{eq:contraction_intermediate}
\begin{align}
    W_{p,t}^p(\cS(\Psi^1),\cS(\Psi^2)) &\le \cK \,t\, W_p^p(\Psi^1,\Psi^2), \label{eq:contraction_intermediate_1} \\
    W_{p,t}^p(\cS^n(\Psi^1),\cS^n(\Psi^2)) &\le \cK \int_0^t W_{p,s}^p(\cS^{n-1}(\Psi^1),\cS^{n-1}(\Psi^2))\,\de s; \label{eq:contraction_intermediate_n}
\end{align}
\end{subequations}
here, $\cS^n$ denotes the composition of $\cS$ with itself $n$ times. Therefore, by induction on $n$, we obtain the crucial estimate
\begin{equation*}
    W_{p,t}^p(\cS^n(\Psi^1),\cS^n(\Psi^2))
    \le \cK \int_0^t \frac{(\cK\,s)^{n-1}}{(n-1)!} W_p^p(\Psi^1,\Psi^2)\,\de s = \frac{(\cK\,t)^n}{n!}W_p^p(\Psi^1,\Psi^2),
\end{equation*}
whence, by taking $t = T$,
\begin{equation}\label{eq:contraction}
    W_p(\cS^n(\Psi^1),\cS^n(\Psi^2))
    \le \biggl(\frac{(\cK\,t)^n}{n!}\biggr)^{1/p}W_p(\Psi^1,\Psi^2).
\end{equation}
Since the quantity $((\cK T)^n/n!)^{1/p}<1$ for $n$ sufficiently large, estimate~\eqref{eq:contraction} implies that $\cS^n$ is a contraction on the complete metric space $\bigl(\cP_p(C([0,T],\real^d)),W_p\bigr)$. Hence, a classical variant of the Banach--Caccioppoli fixed-point theorem (see the original paper \cite{Bryant1968IteratedFixedPoint}, or again the review \cite[Theorem 5]{BanachHistory2024}) provides the existence of a unique fixed point $\rho \in \cP_p(C([0,T],\real^d))$ for $\cS$. As previously mentioned in Section~\ref{subsec:ClassicalSznitman}, the process $X^\rho$ solving the auxiliary SDE \eqref{eq:AuxProblem} associated to $\rho$ is a strong solution to the original McKean--Vlasov SDE \eqref{eq:McKean-Vlasov}, thus proving the existence result. Moreover, since $X^\rho$ solves a particular instance of \eqref{eq:AuxProblem}, condition \eqref{eq:SznitmanFiniteMoments} is automatically satisfied.

Now we prove pathwise uniqueness in the class of strong solutions to \eqref{eq:McKean-Vlasov} satisfying condition \eqref{eq:SznitmanFiniteMoments}. Let us consider two processes $\barX_1$, $\barX_2$ belonging to this class.
Through a computation completely analogous to \eqref{eq:HoldBDGLip}, we obtain that
\begin{equation}\label{eq:intermediateUniqueness}
    \E\biggl[\sup_{u\in[0,t]} \norm{\barX_{1,u} - \barX_{2,u}}^p\biggr] \le \cK_0 \E\biggl[\int_0^t \bigl(\norm{\barX_{1,u} - \barX_{2,u}}^p + W_p^p(\rho_{1,s},\rho_{2,s})\bigr)\,\de s\biggr]
\end{equation}
for all $t \in [0,T]$. Now, by virtue of the McKean--Vlasov condition $\rho_{i,t} = \Law(\barX_{i,t})$, we can deduce that
\begin{equation*}
    W_p^p(\rho_{1,s},\rho_{2,s}) \le \E\Bigl[\norm{\barX_{1,s} - \barX_{2,s}}^p\Bigr], 
\end{equation*}
hence, from \eqref{eq:intermediateUniqueness} we obtain, for all $t \in [0,T]$,
\begin{equation}
    \E\biggl[\sup_{u\in[0,t]} \norm{\barX_{1,u} - \barX_{2,u}}^p\biggr] \le 2\cK_0 \int_0^t \E\biggl[\sup_{u\in[0,s]} \norm{\barX_{1,u} - \barX_{2,u}}^p\biggr]\,\de s.
\end{equation}
Since condition \eqref{eq:SznitmanFiniteMoments} ensures that the map $t \mapsto \E\Bigl[\sup_{u\in[0,t]} \norm{\barX_{1,u} - \barX_{2,u}}^p\Bigr]$ is bounded over $[0,T]$, we can apply Gr\"{o}nwall inequality to conclude that
\begin{equation*}
    \E\biggl[\sup_{u\in[0,T]} \norm{\barX_{1,u} - \barX_{2,u}}^p\biggr] = 0,
\end{equation*}
and pathwise uniqueness follows. The theorem is completely proved.
\hfill    
\end{proof}
\smallskip
\noindent\textbf{Acknowledgments.} The author is a member of the GNAMPA group of INdAM (Istituto Nazionale di Alta Matematica).

\end{document}